\documentclass[11pt,a4paper]{article}
\usepackage[T1]{fontenc}
\usepackage{lmodern}
\usepackage[utf8]{inputenc}
\usepackage[margin=2.4cm]{geometry}
\usepackage{amsmath,amssymb,amsthm}
\usepackage{enumitem}
\usepackage{microtype}
\usepackage[hidelinks]{hyperref}
\hypersetup{pdftitle={Antidirected forests in digraphs},
  pdfauthor={Gengtao Liu and Yunshu Gao}}
\allowdisplaybreaks[2]
\newtheorem{theorem}{Theorem}[section]
\newtheorem{lemma}[theorem]{Lemma}
\newtheorem{proposition}[theorem]{Proposition}
\newtheorem{corollary}[theorem]{Corollary}

\theoremstyle{definition}

\theoremstyle{remark}
\newtheorem{remark}[theorem]{Remark}
\DeclareMathOperator{\ex}{ex}
\newcommand{\dK}{\vec K_2}
\newcommand{\cF}{\mathcal F_k}
\newcommand{\IH}{\textup{(IH)}}

\title{\textbf{Antidirected forests in digraphs}}
\author{Gengtao Liu\quad Yunshu Gao\thanks{Corresponding author: gysh2004@gmail.com}\\
\small School of Mathematics and Statistics, Ningxia University,\\
\small Yinchuan, 750021, China}
\date{}

\begin{document}
\maketitle

\begin{abstract}
A digraph is antidirected if every vertex has indegree zero or outdegree zero.
Let $k\ge2$, and let $F$ be an antidirected forest with $k$ arcs and no isolated vertices.
We prove that every digraph $D$ of order $n$ with more than
\[
 g_k(n):=2\max\left\{\binom{2k-1}{2},\ (k-1)\left(n-\frac{k}{2}\right)\right\}
\]
arcs contains $F$ as a subdigraph. For $n\ge2k-1$, this threshold equals
$2\ex(n,kK_2)$ and is attained by symmetric digraphs arising from extremal
$kK_2$-free graphs. Consequently, the maximum directed extremal number over
all such forests is $2\ex(n,kK_2)$. The proof combines a counting inequality
for rooted antidirected forests, embeddings extending vertex-disjoint arcs,
and vertex deletion. In the remaining case, the Gallai--Edmonds decomposition
of the underlying graph gives the required bound on the number of arcs.
\end{abstract}

\section{Introduction}\label{sec:intro}

A digraph is \emph{antidirected} if every vertex is a source or a sink, that is,
if every vertex has indegree zero or outdegree zero. An \emph{antidirected
forest} is an antidirected orientation of a forest. Antidirected trees and
forests play a distinguished role in extremal digraph theory: as observed
by de Bruijn (see~\cite{Burr82}), no non-antidirected orientation of a tree
can be guaranteed in an oriented graph by a linear lower bound on its
number of arcs. Indeed, every subdigraph of a complete bipartite graph
oriented from one part to the other is antidirected.

The starting point is the Erd\H{o}s--S\'os conjecture~\cite{ErdosSos63}, which
asserts that every graph of order $n$ with more than $(k-1)n/2$ edges
contains every tree with $k$ edges. Brandt~\cite{Brandt94} proved the
corresponding statement for forests without isolated vertices: a graph
with more than
\[
 \max\left\{\binom{2k-1}{2},\ (k-1)\left(n-\frac{k}{2}\right)\right\}
\]
edges contains every forest with $k$ edges and no isolated vertices.
For $n\ge2k-1$, the displayed expression is $\ex(n,kK_2)$; for smaller
orders, the hypothesis is vacuous.

The antidirected analogue of the Erd\H{o}s--S\'os conjecture was proposed
by Addario-Berry, Havet, Linhares Sales, Reed and Thomass\'e~\cite{ABHLRT13}:
every digraph of order $n$ with more than $(k-1)n$ arcs contains every
antidirected tree with $k$ arcs. Results for dense digraphs and under
semidegree assumptions appear in~\cite{ST25,SZ24,Stein23}; the tree
case has recently been settled in~\cite{Deb26,SSW26,RS26}.
For general digraphs, the first bound with a constant linear in $k$ is
due to Burr~\cite{Burr82}: every digraph on $n$ vertices with at least
$4kn$ arcs contains every antidirected tree with $k$ arcs, improving
an earlier bound of Graham~\cite{Graham70}, whose constant is
exponentially large in $k$.
The proposed coefficient $k-1$ is necessary in general: a regular
orientation of $K_{2k-2,2k-2}$ with every indegree and outdegree equal to
$k-1$ has $(k-1)n$ arcs and contains no antidirected star with $k$
arcs; see~\cite{ABHLRT13,Burr82}.

We prove the corresponding extremal theorem for antidirected forests.
Write $\dK$ for the digraph consisting of a single arc, and let $\cF$
be the family of antidirected forests with $k$ arcs and no isolated vertices.
The graph $K_2$ remains undirected; thus $kK_2$ denotes an undirected
matching and $k\dK$ denotes $k$ vertex-disjoint arcs.

\begin{theorem}\label{thm:main}
Let $k\ge2$, let $F\in\cF$, and let $D$ be a digraph of order $n$. If
\[
 m(D)>g_k(n):=2\max\left\{\binom{2k-1}{2},\
 (k-1)\left(n-\frac{k}{2}\right)\right\},
\]
then $D$ contains $F$ as a subdigraph. Moreover, for every $n\ge1$,
\[
 \max_{F\in\cF}\ex_{\rm dir}(n,F)=2\ex(n,kK_2).
\]
\end{theorem}

Here $m(D)$ is the number of arcs of $D$, and $\ex_{\rm dir}(n,F)$ is
the maximum number of arcs in an $F$-free digraph of order $n$.
For $n\ge2k-1$, the symmetric digraph obtained from an extremal
$kK_2$-free graph has exactly $g_k(n)$ arcs and contains no $k\dK$.
For $n<2k$, the matching $k\dK$ has more than $n$ vertices, so its
directed extremal number is $n(n-1)=2\ex(n,kK_2)$.
Thus the threshold is sharp for the family $\cF$, although it need not
be sharp for each individual forest.

The proof has three main parts. First, a counting inequality for rooted
forests gives a bound in terms of the number of vertices of the forest.
Together with bounds on indegrees and outdegrees, this settles
$2k+1\le n\le\lceil5k/2-1\rceil$ by constructing vertex-disjoint copies
of the components that have at least two arcs and retaining enough
vertex-disjoint arcs for the other components. Second, when
$n\ge\lceil5k/2\rceil$, a vertex of total degree at most $2k-2$ can be
deleted while preserving the required strict inequality. Finally, in the
remaining case the Gallai--Edmonds decomposition of the underlying graph
either provides a sufficiently large matching or contradicts the lower
bound on $m(D)$.

Section~\ref{sec:prelim} gives the notation, counting lemmas, base cases
and degree reductions. Section~\ref{sec:interval} treats the range
$2k+1\le n\le\lceil5k/2-1\rceil$, and Section~\ref{sec:main-proof}
completes the proof. Auxiliary polynomial inequalities are proved in
Appendix~\ref{app:ineq}.

\section{Notation and preliminaries}\label{sec:prelim}

All digraphs are finite and have neither loops nor parallel arcs with
the same ordered pair of endpoints; opposite arcs are allowed. A pair
of opposite arcs is a \emph{digon}. All undirected graphs are finite
and simple. We use standard graph and digraph terminology
from~\cite{BM08,BJG09}.

\subsection{Notation}

For a digraph $D$, write $V(D)$ and $A(D)$ for its vertex and arc sets,
$n(D)=|V(D)|$, and $m(D)=|A(D)|$. When $D$ is fixed, put $n=n(D)$.
Its underlying simple graph $G(D)$ has an edge $uv$ precisely when at
least one of $uv$ and $vu$ is an arc of $D$. For an undirected graph $G$,
write $e(G)=|E(G)|$ and $N_G(X)=\bigcup_{v\in X}N_G(v)$ for
$X\subseteq V(G)$. For $v\in V(D)$, let
\[
 \begin{aligned}
 N_D^+(v)&=\{w:vw\in A(D)\},& d_D^+(v)&=|N_D^+(v)|,\\
 N_D^-(v)&=\{w:wv\in A(D)\},& d_D^-(v)&=|N_D^-(v)|.
 \end{aligned}
\]
The \emph{total degree} of $v$ is $d_D(v)=d_D^+(v)+d_D^-(v)$.
We write $\delta(D)$ and $\Delta(D)$ for the minimum and maximum total
degrees. A digon contributes two to the total degree at each endpoint.
For $\sigma\in\{+,-\}$, let $\bar\sigma$ be its opposite sign. Then
\[
 w\in N_D^\sigma(v)\quad\Longleftrightarrow\quad
 v\in N_D^{\bar\sigma}(w).
\]
We omit the subscript $D$ only when the digraph is unambiguous.
For $X\subseteq V(D)$, the induced subdigraph on $X$ is $D[X]$, and
$D-X=D[V(D)\setminus X]$. In particular,
\begin{equation}\label{eq:deletion}
 m(D-X)=m(D)-\sum_{v\in X}d_D(v)+m(D[X]).
\end{equation}
The directed complement $\overline D$ has the same vertex set and arc set
$\{uv:u\ne v,\ uv\notin A(D)\}$.

An \emph{embedding} of a digraph $H$ into $D$ is an injective map
$\varphi:V(H)\to V(D)$ preserving every arc. Its image is a \emph{copy}
of $H$. We write $H\subseteq D$ when such a copy exists; this notation
does not require the copy to be induced. A rooted embedding of $(H,r)$
into $(D,v)$ additionally satisfies $\varphi(r)=v$.
For a non-isolated vertex $v$ of an antidirected forest $H$, define
$\sigma_H(v)=+$ when $v$ is a source and $\sigma_H(v)=-$ when $v$ is
a sink. Adjacent vertices have opposite signs. When an isolated root
occurs in a subforest, its sign is inherited from the original forest;
if no larger forest is specified, either sign may be prescribed.

A \emph{matching of arcs} is a subset of $A(D)$ whose members have
pairwise disjoint endpoint sets. Let $\nu(D)$ be its maximum size;
then $\nu(D)=\nu(G(D))$. For a fixed matching of arcs $M$ and
$W\subseteq V(D)$, write
\[
 M(W)=\{uv\in M:\{u,v\}\cap W\ne\varnothing\}.
\]
Thus $M\setminus M(W)$ is a matching in $D-W$, and $|M(W)|\le|W|$.

For $F\in\cF$, write
\begin{equation}\label{eq:forest-decomposition}
 F=P\sqcup\rho\dK,
\end{equation}
where $P$ has no component consisting of a single arc. Put
\[
 \begin{gathered}
 p=m(P)=k-\rho,\qquad c_P=\text{number of components of }P,\qquad
 c=c_P+\rho,\\
 x=k-c=p-c_P,\qquad n_P=n(P)=p+c_P.
 \end{gathered}
\]
Every component of $P$ has at least two arcs, so
$p\ge2c_P$, $c_P\le x$, and $n_P\le3p/2$.
Also $n(F)=k+c$. We shall use
\begin{equation}\label{eq:parameters}
 \begin{gathered}
 t=n-2k,\qquad \varepsilon=p-2c_P\ge0,\qquad
 \zeta=k+c-2=2k-x-2,\\
 c_P=x-\varepsilon,\qquad p=2x-\varepsilon,\qquad
 \rho=k-2x+\varepsilon,\qquad n_P=3x-2\varepsilon.
 \end{gathered}
\end{equation}
The parameter $\varepsilon$ is unrelated to an indegree or outdegree sign.

Finally, write
\begin{equation}\label{eq:branches}
 g_{\rm cl}(k)=(2k-1)(2k-2),\qquad g_{\rm lin}(k,n)=(k-1)(2n-k),
 \qquad g_k(n)=\max\{g_{\rm cl}(k),g_{\rm lin}(k,n)\}.
\end{equation}
For $k\ge2$, the second expression is at least the first precisely when
$n\ge5k/2-1$. We put $g_1(n)=0$. For each fixed positive integer $k$, the maximum in~\eqref{eq:branches}
also defines $g_k(z)$ for real $z$ in purely numerical comparisons.

Some statements below assume that Theorem~\ref{thm:main} holds for all
pairs $(k',n')$ preceding $(k,n)$ in lexicographic order. We denote this
assumption by \IH. The case $k'=1$ is the assertion that a nonempty
digraph contains an arc. Every use of \IH\ is justified by the minimal
counterexample argument in Section~\ref{sec:main-proof}.

\subsection{Counting rooted copies}\label{ss:counting}

Fix a digraph $D$ of order $n\ge1$. Let $\Omega(D)$ be the set of all
linear orderings $\pi=(v_0,\ldots,v_{n-1})$ of $V(D)$, and put
$V_i(\pi)=\{v_0,\ldots,v_i\}$ for $0\le i\le n-1$.
In an indexed list, an interval with lower endpoint larger than its upper
endpoint denotes the empty list. For $\sigma\in\{+,-\}$, define
\[
 \begin{split}
 I_D^\sigma(\pi)&=\{i\in\{1,\ldots,n-1\}:v_i\in N_D^\sigma(v_0)\},\\
 \mathcal M_D^\sigma&=\{(\pi,i):\pi\in\Omega(D),\ i\in I_D^\sigma(\pi)\}.
 \end{split}
\]
For each fixed first vertex, there are $(n-1)!$ orderings. Hence
\begin{equation}\label{eq:marked-count}
 |\mathcal M_D^+|=|\mathcal M_D^-|=m(D)(n-1)! =:\mu_D.
\end{equation}
The two sets need not be equal; only their cardinalities agree.
Also put $\mathcal Z(D)=\Omega(D)\times\{0\}$, so $|\mathcal Z(D)|=n!$.

For a rooted antidirected forest $(H,r)$ having at least one arc, with
root sign $\sigma_H(r)$ prescribed as above, define
\[
 \mathcal R_D(H,r)=
 \{(\pi,i)\in\mathcal M_D^{\sigma_H(r)}:
 (H,r)\hookrightarrow(D[V_i(\pi)],v_0)\},
 \qquad \lambda_D(H,r)=|\mathcal R_D(H,r)|.
\]
Thus $0\le \lambda_D(H,r)\le \mu_D$.
For $1\le i\le n-1$, let
\[
 \tau_i(\pi)=(v_i,v_1,\ldots,v_{i-1},v_0,v_{i+1},\ldots,v_{n-1}).
\]
This transposition preserves $V_i(\pi)$, and
\begin{equation}\label{eq:transposition}
 (\pi,i)\in\mathcal M_D^\sigma
 \quad\Longleftrightarrow\quad
 (\tau_i(\pi),i)\in\mathcal M_D^{\bar\sigma}.
\end{equation}

\begin{lemma}\label{lem:leaf}
Let $F$ be an antidirected forest with at least two arcs and no isolated
vertices. Let $\ell$ be a leaf with neighbour $z$, and put $S=F-\ell$.
If $z$ is non-isolated in $S$, then
\[
 \lambda_D(S,z)\le \lambda_D(F,\ell)+n!.
\]
The two counts use the opposite root signs $\sigma_F(z)$ and
$\sigma_F(\ell)$.
\end{lemma}

\begin{proof}
For each $\pi$, omit the pair of smallest index in
$\{(\pi,i)\in\mathcal R_D(S,z)\}$, if this set is nonempty.
At most $n!$ pairs are omitted. For any remaining $(\pi,i)$, there is
$j<i$ with $(\pi,j)\in\mathcal R_D(S,z)$. A rooted copy of $S$ in
$D[V_j(\pi)]$ extends to a copy of $F$ by mapping $\ell$ to $v_i$:
this vertex is outside $V_j(\pi)$, and its adjacency to $v_0$ has the
required direction. After applying $\tau_i$, the resulting copy is
rooted at the first vertex with root $\ell$. By~\eqref{eq:transposition},
$(\tau_i(\pi),i)\in\mathcal R_D(F,\ell)$. The map
$(\pi,i)\mapsto(\tau_i(\pi),i)$ is injective, proving the inequality.
\end{proof}

\begin{lemma}\label{lem:isolated-arc}
Let $F$ be an antidirected forest having a component with vertex set
$\{\ell,z\}$ and a single arc. Suppose that $H=F-\{\ell,z\}$ has
at least one arc, and put $S=\{z\}\sqcup H$. With the sign of $z$
inherited from $F$,
\[
 \lambda_D(S,z)\le \lambda_D(F,\ell)+n!.
\]
\end{lemma}

\begin{proof}
A pair $(\pi,i)\in\mathcal M_D^{\sigma_F(z)}$ belongs to
$\mathcal R_D(S,z)$ precisely when $H$ embeds in
$D[V_i(\pi)\setminus\{v_0\}]$. Omit the smallest such index for each
ordering. For any remaining index $i$, choose a smaller such index $j$.
The corresponding copy of $H$ avoids $v_0$ and $v_i$. Adjoining the arc
on these two vertices gives a copy of $F$ with $z$ mapped to $v_0$ and
$\ell$ mapped to $v_i$. The same injective map as in
Lemma~\ref{lem:leaf} sends this pair to $\mathcal R_D(F,\ell)$.
\end{proof}

\begin{lemma}\label{lem:root-union}
Let $H_1,H_2$ be subforests of a rooted antidirected forest $(H,r)$,
each having at least one arc, such that
$H_1\cup H_2=H$ and $V(H_1)\cap V(H_2)=\{r\}$.
With the same prescribed sign at $r$ in all three forests,
\[
 \lambda_D(H_1,r)+\lambda_D(H_2,r)\le \mu_D+n!+\lambda_D(H,r).
\]
\end{lemma}

\begin{proof}
Let $\sigma$ be the prescribed sign at $r$, and write
$\mathcal R_j=\mathcal R_D(H_j,r)$ and $\mathcal R=\mathcal R_D(H,r)$.
Since $\mathcal R\subseteq\mathcal R_1$, it suffices to give an injection
\[
 \mathcal R_1\setminus\mathcal R\longrightarrow
 \mathcal Z(D)\,\dot\cup\,
 (\mathcal M_D^\sigma\setminus\mathcal R_2).
\]
For $(\pi,i)\in\mathcal R_1\setminus\mathcal R$, let
\[
 h=\min\{0\le j\le i:(H_1,r)\hookrightarrow(D[V_j(\pi)],v_0)\},
 \qquad b=i-h.
\]
Because $H_1$ has an arc, $h\ge1$. Define
\[
 \pi'=(v_0,v_{h+1},\ldots,v_i,v_1,\ldots,v_h,v_{i+1},\ldots,v_{n-1})
\]
and map $(\pi,i)$ to $(\pi',b)$. If $b=0$, this lies in
$\mathcal Z(D)$. If $b\ge1$, the pair remains $\sigma$-marked because
its vertex in position $b$ is $v_i$. It cannot belong to $\mathcal R_2$:
a rooted copy of $H_2$ there would be disjoint outside $v_0$ from the
copy of $H_1$ in $D[V_h(\pi)]$, giving a rooted copy of $H$ in
$D[V_i(\pi)]$.

To prove injectivity, write $\pi'=(u_0,\ldots,u_{n-1})$. The integer
$h$ is uniquely determined as the least integer $0\le j\le n-1-b$ for which
\[
 (H_1,r)\hookrightarrow
 \bigl(D[\{u_0,u_{b+1},\ldots,u_{b+j}\}],u_0\bigr).
\]
For $j\le h$, this vertex set is $\{v_0,\ldots,v_j\}$, so the least
such $j$ is exactly $h$. Then $i=h+b$, and the original ordering is
uniquely determined. This argument includes $b=0$.
\end{proof}

\begin{theorem}[Rooted tree inequality]\label{thm:tree-count}
Let $T$ be an antidirected tree with $q\ge1$ arcs, and let $r\in V(T)$.
Then
\[
 \mu_D\le \lambda_D(T,r)+(q-1)n!.
\]
\end{theorem}

\begin{proof}
We induct on $q$, for every choice of root. If $q=1$, every marked pair
contains the required rooted arc, so $\lambda_D(T,r)=\mu_D$.
If $q\ge2$ and $r$ is a leaf with neighbour $z$, put $S=T-r$.
The vertex $z$ is non-isolated in $S$, so induction and
Lemma~\ref{lem:leaf}, applied to $F=T$, give
\[
 \mu_D\le \lambda_D(S,z)+(q-2)n!\le \lambda_D(T,r)+(q-1)n!.
\]
Equality~\eqref{eq:marked-count} permits the change of root sign.

If $d_T(r)\ge2$, let $X$ be the vertex set of one component of $T-r$,
and put $T_1=T[X\cup\{r\}]$ and $T_2=T[V(T)\setminus X]$.
Both trees have at least one arc, and their arc counts sum to $q$.
Induction and Lemma~\ref{lem:root-union} yield
\[
 2\mu_D\le \lambda_D(T_1,r)+\lambda_D(T_2,r)+(q-2)n!
 \le \mu_D+\lambda_D(T,r)+(q-1)n!.
\]
Subtracting $\mu_D$ completes the proof.
\end{proof}

For an antidirected forest $H$ and $\sigma\in\{+,-\}$, define
\[
 \begin{split}
 \mathcal O_D^\sigma(H)&=\{(\pi,i)\in\mathcal M_D^\sigma:
 H\hookrightarrow D[V_i(\pi)\setminus\{v_0\}]\},\\
 o_D^\sigma(H)&=|\mathcal O_D^\sigma(H)|.
 \end{split}
\]
These counts require the copy to avoid the first vertex.

\begin{lemma}\label{lem:forget-root}
Let $T$ be an antidirected tree with at least one arc, and let
$\sigma_T(r)=\sigma$. Then
\[
 \lambda_D(T,r)\le o_D^{\bar\sigma}(T)+n!.
\]
\end{lemma}

\begin{proof}
For each ordering $\pi$, at most one pair $(\pi,i)\in\mathcal R_D(T,r)$
can satisfy $T\not\subseteq D[V_{i-1}(\pi)]$: such an index is necessarily
the first index for which $T\subseteq D[V_i(\pi)]$.
For every other pair, the transposition $\tau_i$ gives a pair counted by
$o_D^{\bar\sigma}(T)$, since
\[
 V_i(\tau_i(\pi))\setminus\{(\tau_i(\pi))_0\}=V_{i-1}(\pi).
\]
This map is injective by~\eqref{eq:transposition}. There are at most
$n!$ exceptional pairs. When $n=1$, both counts are zero.
\end{proof}

\begin{lemma}\label{lem:tree-avoid}
If $T$ is an antidirected tree with $q\ge1$ arcs, then for each sign
$\sigma$,
\[
 \mu_D\le o_D^\sigma(T)+qn!.
\]
\end{lemma}

\begin{proof}
Both signs occur in $T$. Choose a root of sign $\bar\sigma$ and apply
Theorem~\ref{thm:tree-count} and Lemma~\ref{lem:forget-root}.
\end{proof}

\begin{lemma}\label{lem:disjoint-union}
Let $H=H_1\sqcup H_2$, where $H_1,H_2$ are nonempty antidirected forests
without isolated vertices. For either sign $\sigma$,
\[
 o_D^\sigma(H_1)+o_D^\sigma(H_2)\le \mu_D+n!+o_D^\sigma(H).
\]
\end{lemma}

\begin{proof}
Since $\mathcal O_D^\sigma(H)\subseteq\mathcal O_D^\sigma(H_1)$,
it is enough to give an injection from their set difference into
$\mathcal Z(D)\dot\cup(\mathcal M_D^\sigma\setminus\mathcal O_D^\sigma(H_2))$.
For a pair $(\pi,i)$ in this difference, let $a\in\{1,\ldots,i\}$ be the least index
for which $D[\{v_1,\ldots,v_a\}]$ contains $H_1$. Put $b=i-a$ and
\[
 \pi'=(v_0,v_{a+1},\ldots,v_i,v_1,\ldots,v_a,v_{i+1},\ldots,v_{n-1}).
\]
If $b=0$, map to $(\pi',0)\in\mathcal Z(D)$. If $b\ge1$, the pair
$(\pi',b)$ is $\sigma$-marked. A copy of $H_2$ avoiding the first vertex
in this initial segment would be disjoint from the copy of $H_1$ on
$\{v_1,\ldots,v_a\}$, a contradiction.

The map is injective, including when $b=0$: from
$\pi'=(u_0,\ldots,u_{n-1})$ and $b$, determine $a$ as the least
integer $1\le j\le n-1-b$ for which $D[\{u_{b+1},\ldots,u_{b+j}\}]$ contains $H_1$.
Then $i=a+b$ and $\pi$ are uniquely determined.
\end{proof}

\begin{lemma}\label{lem:forest-avoid}
For every nonempty antidirected forest $H$ without isolated vertices
and each sign $\sigma$,
\[
 \mu_D\le o_D^\sigma(H)+(n(H)-1)n!.
\]
\end{lemma}

\begin{proof}
Induct on the number of components. The case of one component is
Lemma~\ref{lem:tree-avoid}. For a nontrivial disjoint union
$H=H_1\sqcup H_2$, induction and Lemma~\ref{lem:disjoint-union} give
\[
 \begin{split}
 2\mu_D&\le o_D^\sigma(H_1)+o_D^\sigma(H_2)+(n(H)-2)n!\\
 &\le \mu_D+o_D^\sigma(H)+(n(H)-1)n!.
 \end{split}
\]
\end{proof}

\begin{theorem}\label{thm:weak}
Let $F$ be an antidirected forest with $v\ge2$ vertices and no isolated
vertices. Every digraph $D$ of order $n\ge v$ with
$m(D)>(v-2)n$ contains $F$.
\end{theorem}

\begin{proof}
We prove, for every root $r\in V(F)$, the stronger inequality
\begin{equation}\label{eq:forest-count}
 \mu_D\le \lambda_D(F,r)+(v-2)n!.
\end{equation}
We induct on $v$. For $v=2$, the forest is a single arc and
$\lambda_D(F,r)=\mu_D$.

Suppose $v\ge3$. If $d_F(r)=1$, let $z$ be its neighbour and put
$S=F-r$. If $z$ is non-isolated in $S$, induction and
Lemma~\ref{lem:leaf} give~\eqref{eq:forest-count}.
If $z$ is isolated in $S$, put $H=F-\{r,z\}$. This is a nonempty
forest without isolated vertices, and
$\lambda_D(S,z)=o_D^{\sigma_F(z)}(H)$. Lemmas~\ref{lem:forest-avoid}
and~\ref{lem:isolated-arc} give
\[
 \begin{split}
 \mu_D&\le \lambda_D(S,z)+(v-3)n!\\
 &\le \lambda_D(F,r)+(v-2)n!.
 \end{split}
\]

If $d_F(r)\ge2$, let $X$ be the vertex set of a component of $F-r$
containing a neighbour of $r$. Because the underlying graph is a forest,
this component contains exactly one neighbour of $r$.
Set $F_1=F[X\cup\{r\}]$ and $F_2=F[V(F)\setminus X]$.
Both forests have at least one arc and no isolated vertices, both have
fewer than $v$ vertices, and $n(F_1)+n(F_2)=v+1$.
Induction and Lemma~\ref{lem:root-union} imply
\[
 2\mu_D\le \lambda_D(F_1,r)+\lambda_D(F_2,r)+(v-3)n!
 \le \mu_D+\lambda_D(F,r)+(v-2)n!.
\]
These cases exhaust the possibilities, since $F$ has no isolated vertices.

If $F\not\subseteq D$, then $\lambda_D(F,r)=0$. Dividing
\eqref{eq:forest-count} by $(n-1)!$ gives $m(D)\le(v-2)n$.
\end{proof}

\begin{corollary}\label{cor:weak}
Every digraph $D$ of order $n$ with $m(D)>2(k-1)n$ contains each
$F\in\cF$.
\end{corollary}

\begin{proof}
The hypothesis forces $n\ge2k$, while $n(F)\le2k$.
Thus $(n(F)-2)n\le2(k-1)n$, and Theorem~\ref{thm:weak} applies.
\end{proof}

\subsection{Matchings and threshold differences}

\begin{theorem}[Erd\H{o}s--Gallai~\cite{EG59}]\label{thm:eg}
For integers $k\ge1$ and $n\ge2k-1$,
\[
 \ex(n,kK_2)=\max\left\{\binom{2k-1}{2},\
 (k-1)\left(n-\frac{k}{2}\right)\right\}.
\]
\end{theorem}

\begin{lemma}\label{lem:small-order}
If $k\ge2$ and $n\le2k-1$, then every digraph of order $n$ satisfies
$m(D)\le g_k(n)$.
\end{lemma}

\begin{proof}
We have $m(D)\le n(n-1)\le(2k-1)(2k-2)=g_{\rm cl}(k)\le g_k(n)$.
\end{proof}

\begin{corollary}\label{cor:matching}
For $k\ge1$, if $m(D)>g_k(n)$, then $\nu(D)\ge k$.
\end{corollary}

\begin{proof}
For $k=1$, this is immediate. For $k\ge2$, Lemma~\ref{lem:small-order}
gives $n\ge2k$. Each edge of $G(D)$ corresponds to at most two arcs,
so $e(G(D))\ge m(D)/2>\ex(n,kK_2)$. Theorem~\ref{thm:eg} gives a
matching of size $k$ in $G(D)$. Choosing an arc on each edge gives
$k$ vertex-disjoint arcs of $D$.
\end{proof}

\begin{lemma}[Threshold differences]\label{lem:differences}
Let $k\ge2$ and $n_0\ge2k$. Then
\begin{align}
 g_k(n_0)-g_k(n_0-1)&\le2k-2,\label{eq:diff0}\\
 g_k(n_0)-g_{k-1}(n_0-2)&\ge2n_0+2k-6,\label{eq:diff1}\\
 g_k(n_0)-g_{k-1}(n_0-1)&\ge2n_0-2.\label{eq:diff2}
\end{align}
For $k\ge3$, one also has
\begin{equation}\label{eq:diff3}
 g_k(n_0)-g_{k-2}(n_0-2)\ge4n_0-6.
\end{equation}
\end{lemma}

\begin{proof}
The first inequality follows because $g_{\rm cl}(k)$ is constant in $n_0$ and
$g_{\rm lin}(k,n_0)$ increases by $2k-2$ when $n_0$ increases by one.
For the other inequalities, direct subtraction gives
\begin{align*}
 g_{\rm lin}(k,n_0)-g_{\rm lin}(k-1,n_0-2)&=2n_0+2k-6,\\
 g_{\rm lin}(k,n_0)-g_{\rm lin}(k-1,n_0-1)&=2n_0-2,\\
 g_{\rm lin}(k,n_0)-g_{\rm lin}(k-2,n_0-2)&=4n_0-6.
\end{align*}
Thus the desired bound holds whenever the subtracted threshold is
given by its linear expression, since $g_k(n_0)\ge g_{\rm lin}(k,n_0)$.
This includes every occurrence of $g_1$.

Otherwise the subtracted threshold is strictly larger than its linear
expression, which implies $n_0\le5k/2-1$.
Use $g_k(n_0)\ge g_{\rm cl}(k)$ and
\[
 g_{\rm cl}(k)-g_{\rm cl}(k-1)=8k-10,\qquad
 g_{\rm cl}(k)-g_{\rm cl}(k-2)=16k-28.
\]
In this range,
\[
 \begin{split}
 2n_0+2k-6&\le7k-8\le8k-10,\\
 2n_0-2&\le5k-4\le8k-10,\\
 4n_0-6&\le10k-10\le16k-28\qquad(k\ge3).
 \end{split}
\]
These comparisons prove all the remaining cases.
\end{proof}

\subsection{Base cases}

\begin{lemma}\label{lem:base}
Theorem~\ref{thm:main} holds for $k\le2$.
\end{lemma}

\begin{proof}
For $k=1$, $g_1(n)=0$ and any arc is a copy of $\dK$.
For $k=2$, the forest is either $2\dK$ or one of the two antidirected
orientations of the three-vertex path. The first case follows from
Corollary~\ref{cor:matching}. In the second case, it suffices to find a
vertex with indegree at least two or outdegree at least two, according
to the prescribed orientation. By Lemma~\ref{lem:small-order}, we may
assume $n\ge4$. Since $m(D)>g_2(n)=2(n-1)>n$, both
$\sum_vd_D^-(v)>n$ and $\sum_vd_D^+(v)>n$, so such vertices exist.
\end{proof}

\begin{theorem}\label{thm:twice-order}
Let $k\ge2$. Every digraph $D$ of order $2k$ with
$m(D)>(2k-1)(2k-2)$ contains every $F\in\cF$.
\end{theorem}

\begin{proof}
Let $v=n(F)\le2k$, and let $\mathcal I$ be the set of injections from
$V(F)$ to $V(D)$. Then $|\mathcal I|=(2k)!/(2k-v)!$ and
\[
 m(\overline D)=2k(2k-1)-m(D)\le4k-3.
\]
For $\varphi\in\mathcal I$, let $b(\varphi)$ be the number of arcs
$uw\in A(F)$ for which $\varphi(u)\varphi(w)\in A(\overline D)$.
For each fixed arc of $F$ and arc of $\overline D$, exactly
$(2k-2)!/(2k-v)!$ injections map the former to the latter. Therefore
\[
 \begin{split}
 \sum_{\varphi\in\mathcal I}b(\varphi)
 &=k\,m(\overline D)\frac{(2k-2)!}{(2k-v)!}\\
 &\le k(4k-3)\frac{(2k-2)!}{(2k-v)!}
 <|\mathcal I|,
 \end{split}
\]
where the last inequality uses
$(2k)(2k-1)-k(4k-3)=k>0$.
Since each $b(\varphi)$ is a nonnegative integer, some injection has
$b(\varphi)=0$ and is an embedding of $F$.
\end{proof}

\begin{lemma}\label{lem:no-single-arcs}
If $m(D)>g_k(n)$ and $F\in\cF$ has no component isomorphic to $\dK$,
then $F\subseteq D$.
\end{lemma}

\begin{proof}
Every component of $F$ has at least two arcs, so
$n(F)=k+c\le3k/2$. Also $n\ge2k$ by Lemma~\ref{lem:small-order}.
If $F\not\subseteq D$, Theorem~\ref{thm:weak} gives
$m(D)\le(3k/2-2)n$. We show that this is at most $g_k(n)$.
When $n\ge5k/2-1$,
\[
 g_{\rm lin}(k,n)-(3k/2-2)n=\frac{k}{2}\bigl(n-2(k-1)\bigr)\ge0.
\]
When $n\le5k/2-1$, the coefficient $3k/2-2$ is nonnegative and
\[
 g_{\rm cl}(k)-(3k/2-2)(5k/2-1)=\frac{k^2}{4}+\frac{k}{2}>0.
\]
Both cases contradict $m(D)>g_k(n)$.
\end{proof}

\subsection{Vertex deletion and degree conditions}

\begin{lemma}\label{lem:delete-one}
If $k\ge2$, $n\ge2k$ and $m(D)>g_k(n)$, then every vertex $v$ satisfies
\[
 m(D-v)>g_{k-1}(n-1).
\]
\end{lemma}

\begin{proof}
Use $d_D(v)\le2n-2$ and~\eqref{eq:diff2} in
$m(D-v)=m(D)-d_D(v)$.
\end{proof}

\begin{lemma}\label{lem:delete-low-degree}
Put $\theta_k(n)=g_k(n)-g_k(n-1)$. If $m(D)>g_k(n)$ and
$d_D(v)\le\theta_k(n)$, then $m(D-v)>g_k(n-1)$.
\end{lemma}

\begin{proof}
We have $m(D-v)>g_k(n)-\theta_k(n)=g_k(n-1)$.
\end{proof}

\begin{proposition}\label{prop:large-one-sided}
Assume \IH, let $m(D)>g_k(n)$ with $n\ge2k$, and let $F\in\cF$
have a component isomorphic to $\dK$. If some $v\in V(D)$ and
$\sigma\in\{+,-\}$ satisfy $d_D^\sigma(v)\ge k+c-1$, then $F\subseteq D$.
\end{proposition}

\begin{proof}
Delete a single-arc component from $F$ to obtain $F'$, with $k-1$ arcs
and $k+c-2$ vertices. By Lemma~\ref{lem:delete-one} and \IH, $D-v$
contains a copy of $F'$. Its vertex set has size $k+c-2$, so at least
one vertex of $N_D^\sigma(v)$ lies outside this copy. The arc on that
vertex and $v$ supplies the deleted component, in the appropriate direction.
\end{proof}

\begin{proposition}\label{prop:max-degree}
Assume \IH, let $m(D)>g_k(n)$ with $n\ge2k$, and let $F\in\cF$
have a component isomorphic to $\dK$.
If $\Delta(D)\le n+k-3$, then $F\subseteq D$.
\end{proposition}

\begin{proof}
Choose any arc $uv\in A(D)$. By~\eqref{eq:deletion}, deleting $u,v$
removes
\[
 d_D(u)+d_D(v)-m(D[\{u,v\}])\le2\Delta(D)-1\le2n+2k-7
\]
arcs. Consequently, by~\eqref{eq:diff1},
\[
 m(D-\{u,v\})>g_k(n)-(2n+2k-7)\ge g_{k-1}(n-2)+1.
\]
Apply \IH\ to embed $F$ minus one single-arc component in
$D-\{u,v\}$, and use $uv$ for that component.
\end{proof}

\section{The range near twice the number of arcs}\label{sec:interval}

Throughout this section assume \IH, $m(D)>g_k(n)$, and
\begin{equation}\label{eq:order-interval}
 2k+1\le n\le\left\lceil\frac{5k}{2}-1\right\rceil.
\end{equation}
Let $F\in\cF$ and use the notation of~\eqref{eq:forest-decomposition}
and~\eqref{eq:parameters}. By Corollary~\ref{cor:matching} and
Lemma~\ref{lem:no-single-arcs}, it suffices to consider
\begin{equation}\label{eq:mixed-forest}
 \rho\ge1,\qquad c_P\ge1.
\end{equation}
In particular, $x\ge c_P\ge1$.

\subsection{Consequences of the degree bounds}

\begin{lemma}\label{lem:weak-case}
If $\zeta n=(k+c-2)n\le g_k(n)$, then $F\subseteq D$.
\end{lemma}

\begin{proof}
We have $n(F)=k+c\le2k<n$ and
$m(D)>g_k(n)\ge(n(F)-2)n$. Apply Theorem~\ref{thm:weak}.
\end{proof}

\begin{lemma}\label{lem:residual-parameters}
Suppose $\zeta n>g_k(n)$, and put
$\eta=(t+1)(t+2)/(2k+t)$. Then
\[
 x\le t-\lfloor\eta\rfloor,
 \qquad 1\le t\le\left\lceil\frac{k}{2}\right\rceil-1.
\]
In particular, if $t=1$, then $x=1$, so $P$ is a two-arc antidirected tree.
\end{lemma}

\begin{proof}
Since $g_k(n)\ge g_{\rm cl}(k)$ and $\zeta=2k-x-2$,
\[
 2k-x-2>\frac{(2k-1)(2k-2)}{2k+t}
 =2k-t-3+\frac{(t+1)(t+2)}{2k+t}.
\]
Thus $x<t+1-\eta$, which gives the asserted integer bound.
The upper bound on $t$ follows from~\eqref{eq:order-interval}.
For $t=1$, $k\ge3$ and $\eta=6/(2k+1)<1$, so $1\le x\le1$.
As $x=p-c_P$ and $p\ge2c_P\ge2$, this forces $c_P=1$ and $p=2$.
\end{proof}

\begin{lemma}\label{lem:degree-consequences}
Suppose $F\not\subseteq D$. Then, for every $v\in V(D)$ and either
sign $\sigma$,
\begin{equation}\label{eq:degree-consequences}
 d_D^\sigma(v)\le \zeta,\qquad
 \Delta(D)\ge n+k-2,\qquad
 n\le k+2c-2.
\end{equation}
In particular, $2x\le k-t-2$.
\end{lemma}

\begin{proof}
The first two inequalities follow from Propositions~\ref{prop:large-one-sided}
and~\ref{prop:max-degree}. Hence
$n+k-2\le\Delta(D)\le2\zeta=2k+2c-4$, giving the third inequality.
Substitute $n=2k+t$ and $c=k-x$ for the final assertion.
\end{proof}

Put
\begin{equation}\label{eq:minimum-arcs}
 m_0=g_k(n)+1,\qquad \lambda=n\zeta-m_0,
 \qquad \xi=m_0-2\zeta n_P+p-g_\rho(n-n_P).
\end{equation}
Whenever $\zeta n>g_k(n)$, the integer $\lambda$ is nonnegative.
Under~\eqref{eq:degree-consequences}, $m_0\le m(D)\le n\zeta$ as well.

\begin{proposition}\label{prop:delete-P}
Suppose~\eqref{eq:degree-consequences} holds. If $\xi>0$, then $F\subseteq D$.
\end{proposition}

\begin{proof}
Since $p\le k$, $n\ge2k$ and
$\ex(n,pK_2)\le\ex(n,kK_2)$, we have $g_p(n)\le g_k(n)$.
Lemma~\ref{lem:no-single-arcs} gives a copy of $P$ in $D$; let $W$
be its vertex set. By~\eqref{eq:deletion},
\[
 m(D-W)\ge m(D)-2\zeta|W|+m(D[W])
 \ge m_0-2\zeta n_P+p>g_\rho(n-n_P).
\]
Also
\[
 (n-n_P)-2\rho=t+x>0.
\]
Corollary~\ref{cor:matching} therefore provides $\rho$ vertex-disjoint
arcs in $D-W$. Together with the copy of $P$, they give $F$.
\end{proof}

\subsection{Extending arcs to vertex-disjoint trees}\label{ss:extensions}

Assume~\eqref{eq:degree-consequences}. Fix a matching of arcs $M$ of
size $k$, which exists by Corollary~\ref{cor:matching}. For an integer
$1\le h\le \zeta$, define
\[
 \begin{split}
 S_h^\sigma&=\{v\in V(D):d_D^\sigma(v)<h\},\\
 M_h^+&=\{uw\in M:d_D^+(u)\ge h\},\qquad
 M_h^-=\{uw\in M:d_D^-(w)\ge h\},\\
 s(h)&=\frac{\lambda}{\zeta-h+1}.
 \end{split}
\]
The denominator is positive. Summing indegrees or outdegrees gives
\begin{equation}\label{eq:small-degrees}
 |S_h^\sigma|\le s(h),\qquad |M_h^\sigma|\ge k-s(h).
\end{equation}
Indeed, $m(D)\le(h-1)|S_h^\sigma|+\zeta(n-|S_h^\sigma|)$, while each arc
of $M\setminus M_h^\sigma$ has a distinct endpoint in $S_h^\sigma$.

\begin{proposition}[Components with two arcs]\label{prop:two-arc-components}
Suppose $\varepsilon=0$. If
\begin{equation}\label{eq:two-arc-conditions}
 \zeta\ge3c_P-1,\qquad
 (k-2c_P)(\zeta-3c_P+2)\ge \lambda,
\end{equation}
then $F\subseteq D$.
\end{proposition}

\begin{proof}
Put $h=3c_P-1$. By~\eqref{eq:small-degrees} and
\eqref{eq:two-arc-conditions}, both $M_h^+$ and $M_h^-$ have at least
$2c_P$ arcs. Every component of $P$ is an antidirected tree on three
vertices, with either a source or a sink as its centre.

Construct disjoint copies of these $c_P$ components successively.
Suppose that $j-1$ components have been embedded and let $W$ be the set
of their $3(j-1)$ images. Each such copy contains an arc of $M$ and
one further vertex, so $|M(W)|\le2(j-1)$. Choose an arc
$uw\in M_h^\sigma\setminus M(W)$, where $\sigma$ is the sign of the
centre of the next component. Such an arc exists since
$2c_P-2(j-1)\ge2$.
If $\sigma=+$, map the centre to $u$ and one leaf to $w$; if
$\sigma=-$, map the centre to $w$ and one leaf to $u$.
In the corresponding neighbourhood of the centre image, at most
$3(j-1)+1$ vertices are forbidden: the vertices in $W$ and the image
of the first leaf. Since
$h=3c_P-1\ge3(j-1)+2$, there is a distinct image for the other leaf.

The resulting copy of $P$ has vertex set $W'$ satisfying
$|M(W')|\le2c_P=p$. Hence at least $k-p=\rho$ arcs of $M$ lie in
$D-W'$, and these complete the embedding of $F$.
\end{proof}

\begin{proposition}[Components with at least three arcs]\label{prop:larger-components}
Suppose $\varepsilon\ge1$. Set
$\Phi=3c_P+\varepsilon+2$. Suppose that an integer $\Lambda$ satisfies
$1\le\Lambda\le \zeta$ and
\begin{align}
 \Lambda&\ge s(\Lambda)+4\varepsilon+1,\label{eq:extension-i}\\
 \zeta&\ge\Phi,\label{eq:extension-ii}\\
 k-s(\Theta)&\ge p,\qquad \Theta=\max\{\Lambda,\Phi\}.
 \label{eq:extension-iii}
\end{align}
Then $F\subseteq D$.
\end{proposition}

\begin{proof}
Let $\mathcal Q$ be the set of components of $P$ with at least three
arcs, and put $b=|\mathcal Q|$. For $Q\in\mathcal Q$, write $q_Q=m(Q)$.
Then
\begin{equation}\label{eq:larger-orders}
 b\le\varepsilon,\qquad
 \sum_{Q\in\mathcal Q}q_Q=2b+\varepsilon,\qquad
 \sum_{Q\in\mathcal Q}n(Q)=3b+\varepsilon\le4\varepsilon.
\end{equation}
By~\eqref{eq:small-degrees} and~\eqref{eq:extension-iii}, each of
$M_\Theta^+$ and $M_\Theta^-$ has at least $p$ arcs.

Embed the components in $\mathcal Q$ first, followed by the two-arc
components. Each component $Q$ will contain a specified arc of $M$;
its remaining $q_Q-1$ vertices then meet at most $q_Q-1$ further arcs
of $M$. Thus a completed copy of $Q$ meets at most $q_Q$ arcs of $M$.
Before any new component is begun, the total number of arcs in completed
components is less than $p$. Therefore, if $W$ is the set of vertices
already used, $|M(W)|<p$ and either prescribed set $M_\Theta^\sigma$
contains an arc outside $M(W)$.

Consider $Q\in\mathcal Q$. Choose a leaf $\ell$ with neighbour $z$.
Take an arc $uw\in M_\Theta^{\sigma_Q(z)}\setminus M(W)$.
If $z$ is a source, set $\varphi(z)=u$ and $\varphi(\ell)=w$;
if $z$ is a sink, set $\varphi(z)=w$ and $\varphi(\ell)=u$.
In both cases $d_D^{\sigma_Q(z)}(\varphi(z))\ge\Lambda$.
Root the underlying tree at $z$ and order its other vertices so that
each parent precedes its children, with $\ell$ already embedded.

Suppose that $y$ is the next vertex and $a$ is its parent. The image
$\varphi(a)$ has degree at least $\Lambda$ in direction $\sigma_Q(a)$:
this holds initially for $z$ and is imposed whenever any further
non-leaf vertex is embedded. If $y$ is not a leaf, choose its image in
\begin{equation}\label{eq:extension-neighbourhood}
 N_D^{\sigma_Q(a)}(\varphi(a))
 \setminus\bigl(W\cup S_\Lambda^{\sigma_Q(y)}\bigr),
\end{equation}
where $W$ now includes all images chosen so far. If $y$ is a leaf,
omit the set $S_\Lambda^{\sigma_Q(y)}$ from the exclusion.
During the embedding of components in $\mathcal Q$, at most
$4\varepsilon$ vertices have been used, by~\eqref{eq:larger-orders}.
Thus~\eqref{eq:extension-i} and~\eqref{eq:small-degrees} show that the
set in~\eqref{eq:extension-neighbourhood} has size at least
$\Lambda-4\varepsilon-s(\Lambda)\ge1$.
Because adjacent vertices have opposite signs, every newly added arc
has the required direction. This constructs all copies of components
in $\mathcal Q$ with pairwise disjoint vertex sets.

For a remaining two-arc component, choose an arc of
$M_\Theta^\sigma\setminus M(W)$, where $\sigma$ is the sign of its
centre, and embed the centre and one leaf on this arc. The corresponding
neighbourhood of the centre image has size at least $\Theta\ge\Phi$.
The total order of $P$ is $n_P=3c_P+\varepsilon$; after placing this
arc, at most $n_P-1$ vertices have been used, including the centre.
At most $n_P-2$ vertices of the relevant neighbourhood are therefore
forbidden. Since $\Phi=n_P+2$, another leaf image exists outside all
previous images. Each such component meets at most two arcs of $M$.

Let $W'$ be the vertex set of the resulting copy of $P$. The estimates
above give
\[
 |M(W')|\le\sum_{Q\in\mathcal Q}q_Q+2(c_P-b)
 =2c_P+\varepsilon=p.
\]
Hence $M\setminus M(W')$ contains at least $k-p=\rho$ arcs, which
supply the remaining components of $F$.
\end{proof}

\begin{remark}\label{rmk:matching-intersections}
The copies of the components constructed above are vertex-disjoint,
and their specified arcs of $M$ are distinct. An additional arc of $M$
may nevertheless meet two different copies, one at each endpoint.
This does not affect the argument: $M(W')$ counts each such arc only
once, and its cardinality is at most the sum of the individual bounds.
No assumption is made that every vertex of $D$ is covered by $M$.
\end{remark}

\subsection{The numerical hypotheses}\label{ss:numerical}

It remains to verify the hypotheses of the two extension propositions.
Assume $\zeta n>g_k(n)$ and~\eqref{eq:degree-consequences}. The parameters
satisfy
\begin{equation}\label{eq:numerical-domain}
 \begin{gathered}
 x\ge1,\qquad 0\le\varepsilon\le x-1,\qquad k\ge2x+3,\\
 1\le t\le\min\left\{\left\lceil\frac{k}{2}\right\rceil-1,
 k-2x-2\right\},\qquad \lambda\ge0.
 \end{gathered}
\end{equation}
In particular, $\rho=k-2x+\varepsilon\ge\varepsilon+3$.
Put
\begin{equation}\label{eq:Psi}
 \Psi=\frac{5k}{2}\zeta-g_{\rm cl}(k)-1.
\end{equation}
The following bound accounts for the change between the two expressions
in~\eqref{eq:branches}.

\begin{lemma}\label{lem:R-upper}
Under~\eqref{eq:numerical-domain} and~\eqref{eq:order-interval},
\[
 0\le \lambda\le\Psi-\zeta=k^2-\frac52kx-k+x-1.
\]
\end{lemma}

\begin{proof}
If $n\le5k/2-1$, then $m_0=g_{\rm cl}(k)+1$ and
$\lambda=n\zeta-g_{\rm cl}(k)-1\le\Psi-\zeta$.
The only remaining possibility is that $k$ is odd and
$n=(5k-1)/2$. In that case
$g_k(n)=g_{\rm cl}(k)+(k-1)$, and
\[
 \lambda=\Psi-\frac \zeta2-(k-1)=\Psi-\zeta-\frac x2\le\Psi-\zeta.
\]
The displayed polynomial follows by expansion.
\end{proof}

\begin{lemma}\label{lem:conditions-zero}
If $\varepsilon=0$, the inequalities in~\eqref{eq:two-arc-conditions} hold.
\end{lemma}

\begin{proof}
Here $c_P=x$ and $\rho=k-2x$. As $k\ge2x+3$,
\[
 \zeta-(3c_P-1)=2k-4x-1\ge5.
\]
The second inequality in~\eqref{eq:two-arc-conditions} is
$2(k-2x)^2\ge \lambda$. By Lemma~\ref{lem:R-upper}, it follows from the identity
\begin{equation}\label{eq:zero-square}
 \begin{split}
 2(k-2x)^2-(\Psi-\zeta)
 &=8\left(x-\frac{11k+2}{32}\right)^2
   +\frac{7k^2+84k+124}{128}>0.
 \end{split}
\end{equation}
\end{proof}

\begin{lemma}\label{lem:critical-line}
Assume~\eqref{eq:numerical-domain} and $\varepsilon\ge1$.
If $3x-10\varepsilon\le3$, then $\xi>0$.
Consequently, if $\xi\le0$, then $3x-10\varepsilon\ge4$.
\end{lemma}

\begin{proof}
A proof by explicit polynomial inequalities is given in
Appendix~\ref{app:critical-line}. It uses the larger real upper bound
$t\le(k-1)/2$ and is valid for both parities of $k$.
\end{proof}

\begin{lemma}\label{lem:quadratic}
Assume~\eqref{eq:numerical-domain} and $\varepsilon\ge1$.
Put $\beta=\zeta-4\varepsilon$. If $\xi\le0$, then $\beta>1$ and
\begin{equation}\label{eq:quadratic-gap}
 (\beta-1)^2>4\lambda.
\end{equation}
Consequently, if $(\zeta-4\varepsilon-1)^2<4\lambda$, then $\xi>0$.
\end{lemma}

\begin{proof}
Suppose $\xi\le0$. By Lemma~\ref{lem:critical-line},
$\varepsilon\le(3x-4)/10$, so $x\ge14/3$ and
\[
 \beta-1=2k-x-4\varepsilon-3
 \ge\frac{10k-11x-7}{5}
 \ge\frac{9x+23}{5}>0.
\]
Using Lemma~\ref{lem:R-upper} and squaring these positive quantities,
\[
 \begin{split}
 (\beta-1)^2-4\lambda
 &\ge\left(\frac{10k-11x-7}{5}\right)^2-4(\Psi-\zeta)\\
 &=\frac{10k(3x-4)+121x^2+54x+149}{25}>0.
 \end{split}
\]
This proves~\eqref{eq:quadratic-gap}. The last assertion is its
contrapositive.
\end{proof}

\begin{lemma}\label{lem:integer-threshold}
Let $\zeta$ and $\varepsilon$ be integers with $\varepsilon\ge1$ and
$\zeta>4\varepsilon+1$, let $\lambda\ge0$, and put
$\beta=\zeta-4\varepsilon$. Suppose $(\beta-1)^2>4\lambda$. Then there is an integer
$\Lambda\in[1,\zeta]$ satisfying~\eqref{eq:extension-i}, where
$s(h)=\lambda/(\zeta-h+1)$ for $1\le h\le \zeta$.
Let $\Lambda_0$ be the least such integer, and put
\[
 \Delta_\beta=\beta^2-4\lambda,\qquad
 u_- =\frac{\beta-\sqrt{\Delta_\beta}}2,\qquad
 u_+ =\frac{\beta+\sqrt{\Delta_\beta}}2.
\]
Then
\begin{equation}\label{eq:integer-root}
 \zeta-\Lambda_0+1=\lfloor u_+\rfloor,
 \qquad u_+\ge \beta-\frac{2\lambda}{\beta}.
\end{equation}
If $\lambda>0$, one also has $\lfloor u_+\rfloor\ge\sqrt \lambda-1$.
\end{lemma}

\begin{proof}
For an integer $1\le\Lambda\le \zeta$, set $u=\zeta-\Lambda+1$.
Since $u>0$, inequality~\eqref{eq:extension-i} is equivalent to
\[
 u^2-\beta u+\lambda\le0.
\]
The roots are real, satisfy $0\le u_-\le u_+\le \beta<\zeta$, and have
distance $\sqrt{\Delta_\beta}>\sqrt{2\beta-1}>1$.
If $\lambda>0$, the interval $[u_-,u_+]$ therefore contains a positive integer.
If $\lambda=0$, it is $[0,\beta]$, and $u=1$ suffices.
The largest permissible integer $u$ is $\lfloor u_+\rfloor$, giving
the first equality in~\eqref{eq:integer-root}.

Since $0\le4\lambda/\beta^2\le1$, the inequality
$\sqrt{1-z}\ge1-z$ for $0\le z\le1$ gives
\[
 u_+=\frac \beta2\left(1+\sqrt{1-\frac{4\lambda}{\beta^2}}\right)
 \ge \beta-\frac{2\lambda}{\beta}.
\]
Finally, $u_-u_+=\lambda$ and $u_-\le u_+$ imply $u_+\ge\sqrt \lambda$,
so $\lfloor u_+\rfloor\ge\sqrt \lambda-1$.
\end{proof}

\begin{lemma}\label{lem:conditions-positive}
Assume~\eqref{eq:numerical-domain}, $\varepsilon\ge1$ and $\xi\le0$.
There is an integer $\Lambda\in[1,\zeta]$ satisfying
\eqref{eq:extension-i}--\eqref{eq:extension-iii}.
\end{lemma}

\begin{proof}
By Lemma~\ref{lem:critical-line}, $3x-10\varepsilon\ge4$.
Because $x$ is an integer and $\varepsilon\ge1$, this gives $x\ge5$
and $k\ge13$. Also $\rho\ge\varepsilon+3\ge4$, so
\[
 \zeta-\Phi=2\rho-4\ge0.
\]
Thus~\eqref{eq:extension-ii} holds.
By Lemmas~\ref{lem:quadratic} and~\ref{lem:integer-threshold}, the least
integer $\Lambda_0$ satisfying~\eqref{eq:extension-i} exists.
Let $\Theta=\max\{\Lambda_0,\Phi\}$.
Since $1\le\Theta\le \zeta$ and $k-p=\rho$, it remains to prove
\begin{equation}\label{eq:matching-threshold}
 \lambda\le\rho(\zeta-\Theta+1).
\end{equation}

Suppose first that $\Lambda_0\le\Phi$. Then
$\zeta-\Theta+1=2\rho-3$. A direct completion of the square gives
\begin{equation}\label{eq:Gamma}
 \begin{split}
 \rho(2\rho-3)-\Psi
 &=8\left(x-\frac{11k+16\varepsilon-12}{32}\right)^2\\
 &\quad+\frac{7k^2+160k\varepsilon-248k+240}{128}>0.
 \end{split}
\end{equation}
Indeed, the numerator in the last term is at least
$7k^2-88k+240=7(k-13)^2+94(k-13)+279>0$.
As $\lambda\le\Psi-\zeta\le\Psi$, inequality~\eqref{eq:matching-threshold} follows.

Suppose now that $\Lambda_0>\Phi$. Put $\beta=\zeta-4\varepsilon$ and use
\eqref{eq:integer-root}. If $\lambda=0$, the conclusion is immediate.
If $\lambda>0$, then
\[
 \rho(\zeta-\Lambda_0+1)
 =\rho\lfloor u_+\rfloor
 \ge\rho\left(\beta-\frac{2\lambda}{\beta}-1\right).
\]
Since $\beta>0$, the final expression is at least $\lambda$ whenever
\[
 \lambda(\beta+2\rho)\le\rho \beta(\beta-1).
\]
Lemma~\ref{lem:cubic-positive}, proved in Appendix~\ref{app:cubic}, gives
\[
 \rho \beta(\beta-1)-(\Psi-\zeta)(\beta+2\rho)>0.
\]
Together with $\lambda\le\Psi-\zeta$ and $\beta+2\rho>0$, this proves the desired
inequality. Hence all three hypotheses hold with $\Lambda=\Lambda_0$.
\end{proof}

\subsection{Completion of the proof in this range}

\begin{theorem}\label{thm:interval}
Assume \IH. For $2k+1\le n\le\lceil5k/2-1\rceil$, every digraph
$D$ of order $n$ with $m(D)>g_k(n)$ contains every $F\in\cF$.
\end{theorem}

\begin{proof}
If $F=k\dK$, use Corollary~\ref{cor:matching}; if $\rho=0$, use
Lemma~\ref{lem:no-single-arcs}. We may therefore assume
\eqref{eq:mixed-forest}. If $\zeta n\le g_k(n)$, apply
Lemma~\ref{lem:weak-case}.

Suppose $\zeta n>g_k(n)$ and $F\not\subseteq D$.
Lemmas~\ref{lem:residual-parameters} and~\ref{lem:degree-consequences}
give~\eqref{eq:numerical-domain} and~\eqref{eq:degree-consequences}.
If $\xi>0$, Proposition~\ref{prop:delete-P} gives a contradiction.
Otherwise, when $\varepsilon=0$, Lemma~\ref{lem:conditions-zero} and
Proposition~\ref{prop:two-arc-components} give $F\subseteq D$.
When $\varepsilon\ge1$, Lemma~\ref{lem:conditions-positive} and
Proposition~\ref{prop:larger-components} give the same conclusion.
Thus no counterexample exists.
\end{proof}

\begin{remark}
Theorem~\ref{thm:interval} makes no assumption on the minimum degree or
on the existence of a perfect matching. The degree inequalities used
in its proof are consequences of the assumption $F\not\subseteq D$,
not additional hypotheses of the theorem.
\end{remark}

\section{Proof of the main theorem}\label{sec:main-proof}

\subsection{A matching disjoint from a copy of the larger components}

\begin{lemma}\label{lem:large-matching}
Let $F\in\cF$ satisfy $\rho\ge1$ and $c_P\ge1$, and let
$m(D)>g_k(n)$. If $\nu(D)\ge k+c_P$, then $F\subseteq D$.
\end{lemma}

\begin{proof}
As in Proposition~\ref{prop:delete-P}, $g_p(n)\le g_k(n)$, so
Lemma~\ref{lem:no-single-arcs} gives a copy of $P$ with vertex set $W$.
Choose a maximum matching of arcs $M$ in $D$.
Since $|M(W)|\le|W|=p+c_P$, the digraph $D-W$ contains at least
\[
 |M\setminus M(W)|\ge\nu(D)-(p+c_P)\ge k-p=\rho
\]
vertex-disjoint arcs. These complete the copy of $P$ to a copy of $F$.
\end{proof}

\subsection{The Gallai--Edmonds decomposition}

\begin{lemma}\label{lem:matching-lower-bound}
Let $F\in\cF$ satisfy $\rho\ge1$ and $c_P\ge1$. Suppose that $D$
has order $n$, $m(D)>g_k(n)$, and
\begin{equation}\label{eq:large-order-assumptions}
 \begin{gathered}
 n\ge\left\lceil\frac{5k}{2}\right\rceil,\qquad
 \delta(D)\ge2k-1,\qquad n\le k+2c-2,\\
 d_D^\sigma(v)\le \zeta=k+c-2
 \quad\text{for all }v\in V(D),\ \sigma\in\{+,-\}.
 \end{gathered}
\end{equation}
Then $\nu(D)\ge k+c_P$.
\end{lemma}

\begin{proof}
Put $G=G(D)$ and $\nu=\nu(D)=\nu(G)$. Suppose, to the contrary, that
$\nu\le k+c_P-1$. Because each neighbour in $G$ contributes at most
two to the total degree in $D$,
\begin{equation}\label{eq:underlying-min-degree}
 \delta(G)\ge\left\lceil\frac{2k-1}{2}\right\rceil=k.
\end{equation}

Let $\mathcal D$ be the set of vertices of $G$ left uncovered by at
least one maximum matching, let
$\mathcal A=N_G(\mathcal D)\setminus\mathcal D$, and put
$\mathcal C=V(G)\setminus(\mathcal A\cup\mathcal D)$.
These sets form the Gallai--Edmonds decomposition. Write
\[
 a=|\mathcal A|,\qquad \gamma=|\mathcal C|,\qquad
 d=|\mathcal D|,
\]
and let $q$ be the number of components of $G[\mathcal D]$, of
orders $n_1,\ldots,n_q$.
The Gallai--Edmonds theorem~\cite[Section~3.2]{LP86} gives
\begin{equation}\label{eq:gallai-edmonds}
 2\nu=n+a-q.
\end{equation}
There is no edge between $\mathcal C$ and $\mathcal D$.
Every vertex in a component of $G[\mathcal D]$ has all its neighbours
in that component or in $\mathcal A$; every vertex in $\mathcal C$
has all its neighbours in $\mathcal C\cup\mathcal A$.

\medskip
\noindent\emph{First, $q\ge a+4$.}
By~\eqref{eq:gallai-edmonds} and the assumed upper bound on $\nu$,
\[
 n+a-q-2k+2\le2c_P.
\]
Since $c_P\le x=k-c$ and $n\le k+2c-2$,
\[
 2c_P\le2(k-c)\le3k-n-2.
\]
Thus $2n\le5k-4-a+q$. As $2n\ge5k$, we obtain $q\ge a+4$.

\medskip
\noindent\emph{Second, $a\ge k$.}
Suppose $a\le k-1$. By~\eqref{eq:underlying-min-degree}, every component
of $G[\mathcal D]$ has order at least $k+1-a$. Hence
\[
 (a+4)(k+1-a)\le\sum_{i=1}^q n_i=d\le n-a\le3k-4-a.
\]
The final inequality uses $c\le k-1$, which follows from $x\ge1$.
However,
\[
 (a+4)(k+1-a)-(3k-4-a)=-a^2+(k-2)a+k+8.
\]
This polynomial is concave on $[0,k-1]$, with positive endpoint values
$k+8$ and $9$. It is therefore positive throughout the interval,
a contradiction.

\medskip
\noindent\emph{Third, the arc count is too small.}
The order bounds $5k/2\le n\le3k-4$ imply $k\ge8$.
Put $u=a+\gamma$; then $d=n-u$.
For each component of $G[\mathcal D]$,
\[
 n_i\le d-q+1,
\]
since every other component contains at least one vertex.
The outdegree bound and the absence of edges between distinct such
components yield
\[
 \begin{split}
 m(D)
 &\le a\zeta+\gamma\min\{u-1,\zeta\}
   +d\min\{d-q+a,\zeta\}\\
 &\le u\zeta+(n-u)\min\{n-u-4,\ 2k+2c_P-2-u\}
 =:\psi(u).
 \end{split}
\]
Here $q\ge a+4$ gives the first bound inside the final minimum, while
$q=n+a-2\nu\ge n+a-2k-2c_P+2$ gives the second.
Moreover, $d\ge q\ge a+4\ge k+4$, so
\[
 k\le u\le u^*:=n-k-4.
\]
The two quantities in the minimum defining $\psi$ differ by a constant
independent of $u$. Thus $\psi$ agrees throughout $[k,u^*]$ with a single
quadratic whose leading coefficient is positive. Consequently,
\begin{equation}\label{eq:convex-endpoints}
 \psi(u)\le\max\{\psi(k),\psi(u^*)\}.
\end{equation}

We compare the endpoints directly. Since
$5k/2\le n\le3k-2x-2$ and $1\le c_P\le x$,
\begin{equation}\label{eq:large-parameter-range}
 1\le c_P\le x\le\frac{k}{4}-1.
\end{equation}
Using the second term of the minimum to bound $\psi(k)$,
\[
 \begin{split}
 g_k(n)-1-\psi(k)
 &\ge n(k-2c_P)+2kc_P+kx-2k^2+k-1\\
 &\ge\frac{k^2}{2}-3kc_P+kx+k-1\\
 &\ge\frac{k^2}{2}-2kx+k-1
 \ge3k-1.
 \end{split}
\]
The second inequality uses $n\ge5k/2$ and $k-2c_P>0$.
For $u^*=n-k-4$, use the first term of the minimum to obtain
\[
 \begin{split}
 g_k(n)-1-\psi(u^*)
 &\ge x(n-k-4)+3k-9\\
 &\ge\frac{9k}{2}-13\ge3k-1,
 \end{split}
\]
where $x\ge1$, $n\ge5k/2$ and $k\ge8$.
Together with~\eqref{eq:convex-endpoints}, these inequalities imply
$m(D)\le\psi(u)\le g_k(n)-3k<g_k(n)$, a contradiction.
\end{proof}

\subsection{Completion of the induction}

\begin{proof}[Proof of Theorem~\ref{thm:main}]
Suppose that the embedding assertion fails. Choose a counterexample
$(D,F)$ first with $k$ minimal, and then with $n$ minimal. This gives
the induction hypothesis \IH\ used above.
By Lemmas~\ref{lem:base} and~\ref{lem:small-order} and
Theorem~\ref{thm:twice-order}, we have $k\ge3$ and $n\ge2k+1$.
Corollary~\ref{cor:matching} and Lemma~\ref{lem:no-single-arcs} give
$\rho\ge1$ and $c_P\ge1$.
Theorem~\ref{thm:interval} excludes
$n\le\lceil5k/2-1\rceil$, so $n\ge\lceil5k/2\rceil$.

In this range, both $g_k(n)$ and $g_k(n-1)$ are given by their linear
expressions, and therefore
\[
 g_k(n)-g_k(n-1)=2k-2.
\]
If $D$ had a vertex of total degree at most $2k-2$,
Lemma~\ref{lem:delete-low-degree} would give
$m(D-v)>g_k(n-1)$, and \IH\ would embed $F$ in $D-v$.
Hence $\delta(D)\ge2k-1$.

Propositions~\ref{prop:large-one-sided} and~\ref{prop:max-degree} imply
\[
 d_D^\sigma(v)\le \zeta\quad\text{for every }v,\sigma,
 \qquad \Delta(D)\ge n+k-2.
\]
Since $\Delta(D)\le2\zeta$, this gives $n\le k+2c-2$.
Thus all hypotheses of Lemma~\ref{lem:matching-lower-bound} hold, and
$\nu(D)\ge k+c_P$. Lemma~\ref{lem:large-matching} now embeds $F$ in
$D$, the final contradiction.

For $n\ge2k-1$, the embedding assertion gives
$\ex_{\rm dir}(n,F)\le g_k(n)=2\ex(n,kK_2)$ for every $F\in\cF$.
Equality for the maximum over $\cF$ follows by taking $F=k\dK$ and
the symmetric digraph on an extremal $kK_2$-free graph.
For $n<2k-1$, no digraph of order $n$ contains $k\dK$, so the maximum
is $n(n-1)=2\ex(n,kK_2)$. This proves the stated extremal identity
for every $n\ge1$.
\end{proof}

\section{Concluding remarks}\label{sec:remarks}

The matching $k\dK$ attains the maximum directed extremal number over
$\cF$. Determining $\ex_{\rm dir}(n,F)$ for a fixed antidirected forest
$F$ remains a separate problem: the uniform bound in
Theorem~\ref{thm:main} need not be attained by every forest.

The proof separates the components consisting of a single arc from
those with at least two arcs. The rooted counting inequality guarantees
a copy of the latter part, while the matching arguments guarantee
vertex-disjoint arcs outside it. In the range treated in
Section~\ref{sec:interval}, the copies are chosen to meet at most $p$
arcs of a fixed matching of size $k$. For larger orders, a maximum
matching of size at least $k+c_P$ suffices for any copy of $P$.

There is no corresponding linear bound for all orientations of forests:
a complete bipartite digraph with all arcs directed from one part to
the other has only antidirected subdigraphs and can have a quadratic
number of arcs. It is natural to ask how additional restrictions on
the arcs of the containing digraph affect the extremal numbers for
antidirected forests.

\appendix
\section{Auxiliary polynomial inequalities}\label{app:ineq}

This appendix supplies the two numerical arguments used in
Section~\ref{ss:numerical}. The notation is
\[
 \begin{gathered}
 \zeta=2k-x-2,\qquad \rho=k-2x+\varepsilon,\qquad
 p=2x-\varepsilon,\qquad n_P=3x-2\varepsilon,\\
 \Psi=\frac{5k}{2}\zeta-(2k-1)(2k-2)-1,\qquad
 \xi=g_k(2k+t)+1-2\zeta n_P+p-g_\rho(2k+t-n_P).
 \end{gathered}
\]

\subsection{Proof of Lemma~\ref{lem:critical-line}}\label{app:critical-line}

Assume $k\ge2x+3$, $x\ge2$, $1\le\varepsilon\le x-1$,
$3x-10\varepsilon\le3$, and the upper bounds on $t$ in
\eqref{eq:numerical-domain}. Put
\[
 t^*=\min\left\{\frac{k-1}{2},\ k-2x-2\right\},\qquad
 n_*=2k+t^*-n_P,\qquad
 \varepsilon_0=\frac{3x-3}{10}.
\]
Then $t\le t^*$ and $\varepsilon\ge\varepsilon_0$.
The expression $g_\rho$ is nondecreasing in its argument, so
\begin{equation}\label{eq:L-lower}
 \xi\ge g_{\rm cl}(k)+1-2\zeta n_P+p-g_\rho(n_*).
\end{equation}
Moreover,
\[
 2n_*-5\rho+2=
 \begin{cases}
 k-\varepsilon-2,& t^*=k-2x-2,\\
 4x-\varepsilon+1,& t^*=(k-1)/2.
 \end{cases}
\]
Both expressions are positive. Thus $g_\rho(n_*)$ is given by its
linear expression $(\rho-1)(2n_*-\rho)$.

If $k\le4x+3$, then $t^*=k-2x-2$, and the right-hand side of
\eqref{eq:L-lower} is
\[
 \xi_1(k,x,\varepsilon)=
 -3\varepsilon^2+10\varepsilon x-2\varepsilon
 -k^2+6kx+3k-10x^2-2x-1.
\]
For fixed $x,\varepsilon$, this is concave in $k$ on $[2x+3,4x+3]$.
For $\varepsilon\le x-1$,
\[
 \frac{\partial \xi_1}{\partial\varepsilon}
 =10x-6\varepsilon-2\ge4x+4>0.
\]
Consequently its value is at least the smaller of the two endpoint
values with $\varepsilon$ replaced by $\varepsilon_0$. These values are
\begin{align*}
 \xi_1(2x+3,x,\varepsilon_0)
 &=\frac{73x^2+694x-67}{100}\ge\frac{1613}{100},\\
 \xi_1(4x+3,x,\varepsilon_0)
 &=\frac{73x^2+94x-67}{100}\ge\frac{413}{100}.
\end{align*}
Both polynomials are increasing for $x\ge2$.

If $k\ge4x+3$, then $t^*=(k-1)/2$, and the right-hand side of
\eqref{eq:L-lower} is
\[
 \xi_2(k,x,\varepsilon)=
 -3\varepsilon^2+\varepsilon k+6\varepsilon x-5\varepsilon
 -k-2x^2+8x+2.
\]
Its coefficient of $k$ is $\varepsilon-1\ge0$. Hence
\[
 \begin{split}
 \xi_2(k,x,\varepsilon)
 &\ge \xi_2(4x+3,x,\varepsilon)
 =\xi_1(4x+3,x,\varepsilon)\\
 &\ge \xi_1(4x+3,x,\varepsilon_0)\ge\frac{413}{100}>0.
 \end{split}
\]
This proves $\xi>0$ in both cases and establishes
Lemma~\ref{lem:critical-line}.

\subsection{A cubic inequality}\label{app:cubic}

\begin{lemma}\label{lem:cubic-positive}
Let $k,x,\varepsilon$ be integers satisfying
\[
 k\ge2x+3,\qquad \varepsilon\ge1,\qquad 3x-10\varepsilon\ge4.
\]
With $\beta=\zeta-4\varepsilon$, define
\[
 m_3(k,x,\varepsilon)=\rho \beta(\beta-1)-(\Psi-\zeta)(\beta+2\rho).
\]
Then $m_3(k,x,\varepsilon)>0$.
\end{lemma}

\begin{proof}
Set $\alpha=3x-10\varepsilon-4\ge0$. Direct subtraction gives
\[
 \begin{split}
 m_3(k+1,x,\varepsilon)-m_3(k,x,\varepsilon)
 &=\alpha(2k+1)-\frac72x^2+31\varepsilon x\\
 &\quad+11x+8\varepsilon+8.
 \end{split}
\]
This difference is nondecreasing in $k$. At $k=2x+3$ it equals
\[
 \begin{split}
 \frac{17}{2}x^2+16x-20-\varepsilon(9x+62)
 &\ge\frac{29x^2+5x+24}{5}>0,
 \end{split}
\]
where we used $\varepsilon\le(3x-4)/10$.
Therefore $m_3(k,x,\varepsilon)\ge m_3(2x+3,x,\varepsilon)$.
At this least value of $k$, expansion and factorisation give
\[
 \begin{split}
 m_3(2x+3,x,\varepsilon)
 &=(\varepsilon+3)(3x+4-4\varepsilon)(3x+3-4\varepsilon)\\
 &\quad+\left(x^2-\frac72x-5\right)(3x+10-2\varepsilon).
 \end{split}
\]
All factors are positive. Indeed, the hypotheses imply $x\ge5$, and
\[
 \begin{gathered}
 3x+3-4\varepsilon\ge\frac{9x+23}{5}>0,\qquad
 3x+10-2\varepsilon\ge\frac{12x+54}{5}>0,\\
 x^2-\frac72x-5\ge\frac52>0.
 \end{gathered}
\]
This proves the lemma.
\end{proof}

\section*{Declaration on AI-assisted preparation}
The DeepSeek-Flash model assisted with the technical details and the preparation of the manuscript for this paper.
The authors assume responsibility for all content.

\end{document}